\documentclass[12pt]{article}
\usepackage{pifont}
\usepackage{amsmath}
\usepackage{amssymb,float}
\usepackage{amsthm,graphicx,pifont}
\usepackage{graphicx,psfrag,float,subfigure}
\usepackage{caption,enumerate}
\usepackage[numbers,sort&compress]{natbib}
\usepackage[scale=0.8,a4paper]{geometry}
\usepackage[colorlinks]{hyperref}
\usepackage{amsthm,amsmath,amssymb}
\usepackage{mathrsfs}
\newtheorem {claim}{Claim}
\newtheorem{theorem}{Theorem}[section]

\newtheorem{lemma}[theorem]{Lemma}

\newtheorem{pro}[theorem]{Problem}

\title{Claw-free bricks that every $b$-invariant edge is solitary}
\author{{\small\bf Yipei Zhang$^{1}$, Xiumei Wang$^{2}$}\thanks{Corresponding author. Email address: wangxiumei@zzu.edu.cn}\\
{\small $^{1}$School of Mathematics and Statistics, North China University of Water Resources and Electric Power,}\\
{\small Zhengzhou, Henan 450046,  China}\\
{\small $^{2}$School of Mathematics and Statistics, Zhengzhou University,}\\
{\small Zhengzhou, Henan 450001,  China}}
\date{}

\makeatother
\begin{document}
\maketitle

\begin{abstract}
A graph $G$ is a brick if it is 3-connected and $G-\{u,v\}$ has a perfect matching for any two distinct vertices $u$ and $v$ of $G$.
Lucchesi and Murty proposed a problem concerning the characterization of bricks, distinct from $K_4$, $\overline{C_6}$ and the Petersen graph, in which every $b$-invariant edge is solitary.
In this paper, we present a characterization of this problem when the bricks are claw-free.\\

\end{abstract}

\section{Introduction}

All  graphs considered in this paper are finite and simple.
For any undefined notation and terminology, we follow \cite{BM08}.
Let $G$ be a graph with vertex set $V(G)$ and edge set $E(G)$.
For a subset $X$ of $V(G)$, we denote by $G[X]$ the subgraph of $G$ induced by $X$.
We denote by $K_n$ and $K_{s,t}$ the complete graph with $n$ vertices and the complete bipartite
graph with parts of sizes $s$ and $t$, respectively.
We say that a graph $G$ is {\it claw-free} if it does not contain $K_{1,3}$ as an induced subgraph.

For two disjoint nonempty proper subsets $X$ and $Y$ of $V(G)$,
let $E_G[X,Y]$ denote the set of all the edges of $G$ with one end in $X$ and the other end in $Y$,
let $N_G(X)$ denote the set of all the vertices  in $\overline{X}$ that have a neighbour in $X$, and
let $\partial_G(X)$ denote the set $E_G[X,\overline{X}]$, where $\overline{X}:=V(G)\setminus X$.
The subscript $G$ is omitted from the notation $E_G[X,Y]$, $N_G(X)$ and $\partial_G(X)$ when no ambiguity is possible.
The set $\partial(X)$ is referred to as a {\it cut} of $G$.
For a  cut $\partial(X)$ of $G$, we say  $\partial(X)$ is {\it trivial} if either $X$ or $\overline{X}$
consists of a single vertex, and is \emph{nontrivial} otherwise.
Further, a cut $\partial(X)$ of $G$ is called {\it tight} if each perfect matching of $G$ contains exactly one edge in $\partial(X)$.
For a  cut $\partial(X)$ of $G$,  we use
$G/X$ and $G/\overline{X}$ to denote the two graphs obtained from $G$ by shrinking $X$ and $\overline{X}$ to a single vertex, respectively, and such  two graphs are called the {\it $\partial(X)$-contractions} of $G$.

A connected graph $G$  with at least one edge is \emph{matching covered} if every edge of $G$ is contained in a perfect matching of $G$.
A matching covered graph which is free of nontrivial tight cuts is
called a \emph{brick} if it is nonbipartite, and  a \emph{brace} otherwise.
It is important to observe that if a matching covered graph $G$ has a nontrivial tight cut $\partial(X)$, then the two $\partial(X)$-contractions of $G$ are also matching covered.
By repeatedly performing such contractions, we can derive a list of matching covered graphs without nontrivial tight cuts, which are  either bricks or braces.
This procedure is known as a {\it tight cut decomposition} of $G$.
Note that the tight cut decomposition of a matching covered graph is generally not unique, as it depends on the sequence of nontrivial tight cuts selected during the decomposition process.
However, Lov\'asz \cite{Lovasz1987} proved that any two tight cut decompositions of a matching covered graph yield the same list of bricks and braces, up to multiple edges.
These bricks are called  the bricks of $G$, and their number is denoted by $b(G)$.
Observe that $b(G)=0$ if and only if $G$ is bipartite, and $b(G)=1$ if $G$ is a brick.

For the bricks, Edmonds et al. \cite{ELP82}  established an equivalent definition by demonstrating that
a graph $G$ is a \emph{brick} if and only if it is 3-connected and $G-\{u,v\}$ has a perfect matching for any two distinct vertices $u$ and $v$ of $G$.
Hence, for any brick $G$, we can derive that $|V(G)|\geq4$ and $\delta(G)\geq 3$.

\begin{figure}[h]
 \centering
 \includegraphics[width=0.8\textwidth]{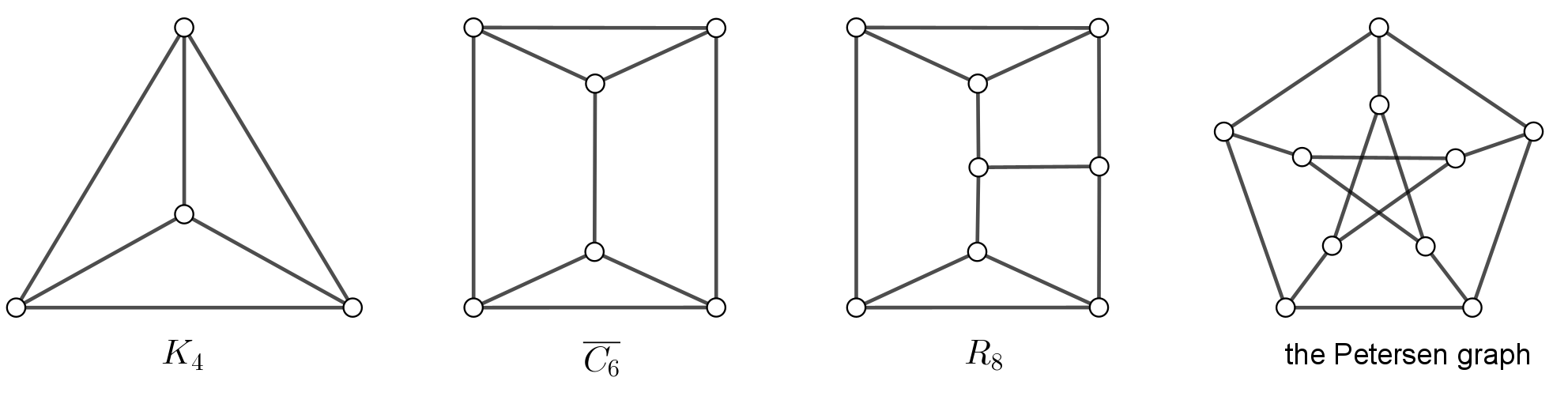}\\
 \caption{The four bricks.}\label{fig1}
\end{figure}

Let $G$ be a matching covered graph.
An edge $e$ of $G$ is \emph{removable} if $G-e$ is also matching covered, and is {\it nonremovable} otherwise.
A removable edge $e$ of $G$ is \emph{$b$-invariant} if $b(G-e)=b(G)$.
For the existence of $b$-invariant edges in bricks, Lov\'asz \cite{Lovasz1987} proposed the
conjecture that every brick different from $K_4$, $\overline{C_6}$ and the Petersen graph
(as shown in Figure \ref{fig1}) has a $b$-invariant edge.
Carvalho et al. \cite{CLM2002} confirmed this conjecture and  improved the result as follows.

\begin{theorem}[\cite{CLM2002}]\label{2-b-invariant-edge}
Every brick different from $K_4$, $\overline{C_6}$, $R_8$ and the Petersen graph
has at least two b-invariant edges.
\end{theorem}

For specific families of bricks, the number of $b$-invariant edges is  related to the number of vertices. For instance,
Carvalho et al. \cite{CLM2012} showed that every solid brick other than $K_4$ has at least $\frac{|V(G)|}{2}$ $b$-invariant edges.
Kothari et al. \cite{KCLL2020} proved that every essentially 4-edge-connected cubic non-near-bipartite brick $G$, other than the Petersen graph, has at least $|V(G)|$ $b$-invariant edges.
They further conjectured that every essentially 4-edge-connected cubic near-bipartite brick $G$,
other than $K_4$, has at least $\frac{|V(G)|}{2}$ $b$-invariant edges.
Lu et al. \cite{LFW2020} confirmed this conjecture and characterized all the  graphs attaining this lower bound, which are prism with $4k+2$ vertices and M\"{o}bius ladder with $4k$ vertices, where $k\geq2$.

An edge of a graph $G$ is {\it solitary} if it is contained in precisely one perfect matching of $G$, and is {\it nonsolitary} otherwise.
Recently, Lucchesi and Murty proposed the following problem, see Unsolved Problems in \cite{LM2024}.

\begin{pro}[\cite{LM2024}]\label{pro}
Characterize bricks, distinct from $K_4$, $\overline{C_6}$ and the Petersen graph, in
which every $b$-invariant edge is solitary.
\end{pro}

\begin{figure}[h]
 \centering
 \includegraphics[width=\textwidth]{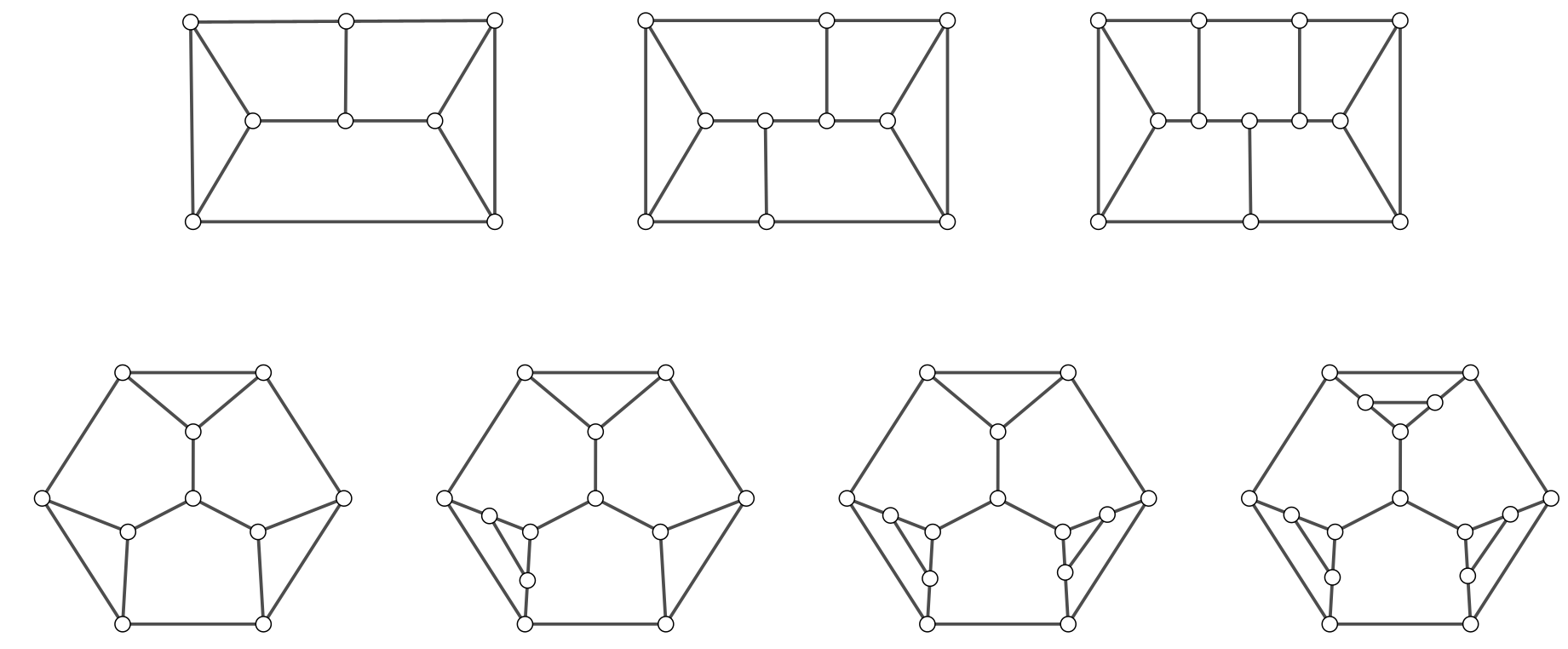}\\
 \caption{The family $\mathcal{F}$.}\label{fig2}
\end{figure}

\begin{figure}[h]
 \centering
 \includegraphics[width=\textwidth]{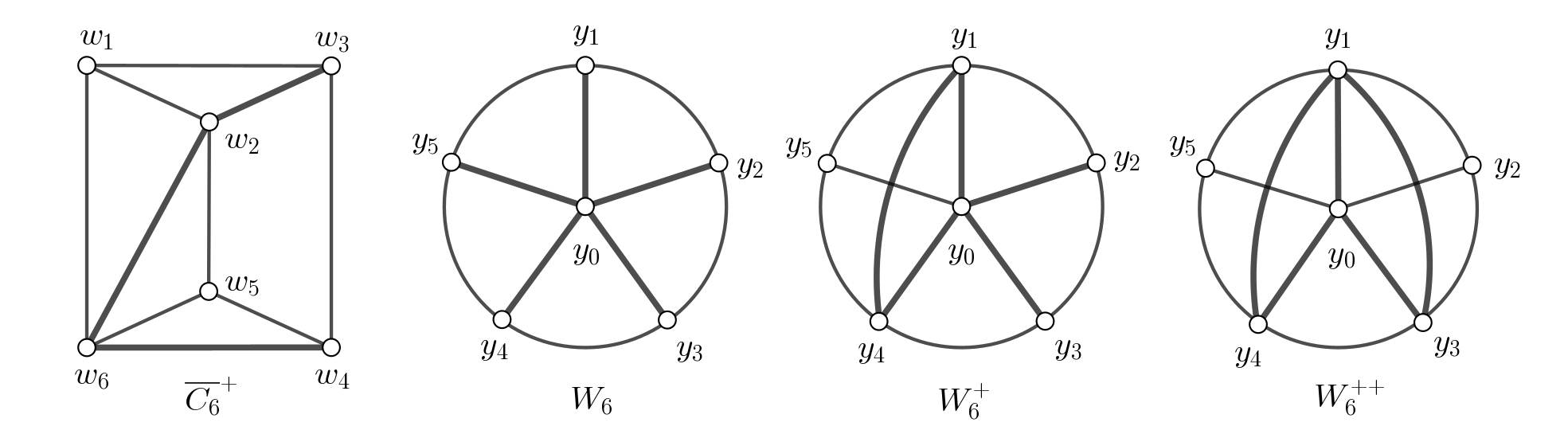}\\
 \caption{The family $\mathcal{G}$.}\label{fig3}
\end{figure}

For the Problem \ref{pro}, Zhang et al. \cite{ZLZ2025+} established that the cubic bricks satisfying the above property are precisely the graphs in family $\mathcal{F}$, as shown in Figure \ref{fig2}.
Note that every graph in $\mathcal{F}$ has a claw.
In this paper, we characterized the case when bricks are claw-free and obtained  all the graphs that belong to $\mathcal{G}$, as shown in Figure \ref{fig3}.
The following is our main result.

\begin{theorem}\label{main-theorem}
Let $G$ be a claw-free brick distinct from $K_4$ and $\overline{C_6}$. Then every $b$-invariant edge of $G$ is solitary if and only if $G\in\mathcal{G}$.
\end{theorem}

We organize the rest of the paper as follows. In Section \ref{section 2}, we present some lemmas. In Section \ref{section 3}, we give a proof of Theorem \ref{main-theorem}.

\section{Preliminaries}\label{section 2}

In the remainder of this paper, unless stated otherwise, we use $N(X)$ and $E[X,Y]$  instead of $N_G(X)$ and $E_G[X,Y]$, respectively, where $X,Y\subseteq V(G)$.
An edge $e$ of a connected  graph $G$ is called a \emph{bridge} if $G-e$ is disconnected.
The following lemma gives an important property of connected graphs admitting exactly one perfect matching.

\begin{lemma} [\cite{Kotzig1959}]\label{connected-with-upm}
If a connected graph $G$ has a unique perfect matching, then it has a bridge belonging to the perfect matching.
\end{lemma}

\begin{lemma}\label{triangle-nonremovable}
Let $G$ be a matching covered graph that contains a triangle $C=u_1u_2u_3u_1$.
If  some vertex $u_i$ of $C$ has exactly one neighbour $v$ in $V(G)\backslash V(C)$, then $u_iv$ is nonremovable in $G$, where $i\in\{1,2,3\}$.
\end{lemma}

\begin{proof}
Assume, without loss of generality, that the vertex $u_1$ has exactly one neighbour $v$ in $V(G)\backslash V(C)$. Then $N(u_1)=\{u_2,u_3,v\}$. Note that every perfect matching of $G$ that contains $u_2u_3$ must contain $u_1v$, which implies that $u_2u_3$ belongs to no perfect matching of $G-u_1v$.
It follows that $G-u_1v$ is not matching covered, and hence $u_1v$ is nonremovable in $G$.
The result holds.
\end{proof}

\begin{lemma}\label{sufficiency}
If $G\in\mathcal{G}$, then every $b$-invariant edge of $G$ is solitary.
\end{lemma}

\begin{proof}
Assume that $G\in\mathcal{G}$. Then $G$ is  claw-free.
Since  $\overline{C_6}$ and $W_6$ are bricks, by the definition of bricks, $G$ is a brick.
It follows that $G$ is matching covered.
Let
\begin{equation}
A=\begin{cases}
\{w_2w_3,w_2w_6,w_4w_6\}, & if\ G\ is\ \overline{C_6}^+\\
\{y_0y_1,y_0y_2,y_0y_3,y_0y_4,y_0y_5\}, & if\ G\ is\ W_6\\
\{y_0y_1,y_0y_2,y_0y_3,y_0y_4,y_1y_4\}, & if\ G\ is\ W_6^+\\
\{y_0y_1,y_0y_3,y_0y_4,y_1y_3,y_1y_4\}, & if\ G\ is\ W_6^{++}.\\
\end{cases}
\nonumber
\end{equation}
Note that each edge of $A$ is solitary in $G$. We shall show that  all the $b$-invariant edges of $G$ are contained in the set $A$, which implies that each $b$-invariant edge of $G$ is solitary. Hence, the result holds.

If $G$ is $\overline{C_6}^+$, then by Lemma \ref{triangle-nonremovable},
each edge of $E(G)\backslash(A\cup\{w_1w_2,w_5w_6\})$ 
is nonremovable in $G$.
Observe that each perfect matching of $G$ that contains $w_1w_2$ must contain $w_5w_6$, and each perfect matching of $G$ that contains $w_5w_6$ must contain $w_1w_2$. This implies that the two edges $w_1w_2$ and $w_5w_6$ are both nonremovable in $G$.
Recall that every $b$-invariant edge  is removable.
Hence, all the $b$-invariant edges of $G$ are contained in the set $A$.

If $G$ is one of the two graphs $W_6$ and $W_6^+$, then again by Lemma \ref{triangle-nonremovable}, each edge of $E(G)\backslash A$ is nonremovable in $G$.
It follows that all the $b$-invariant edges of $G$ are contained in the set $A$.

We now assume that $G$ is $W_6^{++}$.
Then by Lemma \ref{triangle-nonremovable}, each edge of $E(G)\backslash (A\cup\{y_3y_4\})$ is nonremovable in $G$.
Observe that $G'=G-\{y_3y_4\}$ is matching covered and it has a nontrivial tight cut $\partial_{G'}(X)$, where $X=\{y_1,y_4,y_5\}$.
Moreover, the underlying simple graphs of the two $\partial_{G'}(X)$-contractions of $G'$ are both $K_4$.
Hence, we have $b(G')=2$, which implies that the edge $y_3y_4$ is not a $b$-invariant edge of $G$.
It follows that all the $b$-invariant edges of $G$ are contained in the set $A$.
\end{proof}

In fact, one may easily verify that each edge of $A$ is a $b$-invariant edge of $G$. So the set of all the $b$-invariant edges of $G$ is  precisely the set $A$.

\section{Proof of Theorem \ref{main-theorem}}\label{section 3}

Let $G$ be a claw-free brick distinct from $K_4$ and $\overline{C_6}$.
If $G\in\mathcal{G}$, then by Lemma \ref{sufficiency}, every $b$-invariant edge of $G$ is solitary.
The sufficiency holds.

Next, we prove the necessity.  Suppose that every $b$-invariant edge of $G$ is solitary.
We shall show that $G\in\mathcal{G}$.
Since $G$ is a  brick, $G$ is 3-connected.
Further, $G$ is different from $R_8$ (see Figure \ref{fig1}) and the Petersen graph because $G$ is claw-free.
By Theorem \ref{2-b-invariant-edge}, $G$ has two $b$-invariant edges, say one of them is $uv$.
By the assumption, $uv$ is solitary in $G$, implying that $G-\{u,v\}$ has a unique perfect matching, say $M$.
Moreover, the graph $G-\{u,v\}$ is connected since $G$ is 3-connected.
Using Lemma \ref{connected-with-upm}, $G-\{u,v\}$ has a bridge belonging to $M$.
Let $H$ denote the graph obtained from $G-\{u,v\}$ by contracting each   maximal 2-edge-connected subgraph  to a vertex. Then $H$ is a tree with at least two vertices, and each edge of $H$ corresponds to a bridge of $G-\{u,v\}$.
We first show the following claim.

\begin{claim} \label{H-path}
The graph $H$ is a path.
\end{claim}
Assume, to the contrary, that $H$ is not a path. Then $H$ has at least three leaves.
Let $G_1$, $G_2$ and $G_3$ be the three subgraphs of $G-\{u,v\}$ that are contracted to three leaves $v_1$, $v_2$ and $v_3$ of \ $H$.
Then $E[V(G_1),V(G_2)]=E[V(G_1),V(G_3)]=E[V(G_2),V(G_3)]=\emptyset$.
Suppose that one of $u$ and $v$, say $u$, has no neighbour in $G_i$, say $G_1$.
Let $u_1$ be the only neighbour of $v_1$ in $H$ and $e_1$ be the bridge of  $G-\{u,v\}$ that corresponds to $u_1v_1$.
We also use $u_1$ to denote the end of $e_1$ that does not lie in $G_1$.
Then $\{v, u_1\}$ is a vertex cut of $G$, a contradiction to the assumption that $G$ is 3-connected.
So  each of $u$ and $v$ has at least one neighbour in $G_i$, where $i=1,2,3$.
This implies that $G$ contains claws centered at the vertices $u$ and $v$, respectively, a contradiction.
Claim \ref{H-path} holds.
\vspace{2mm}

By Claim \ref{H-path},  $H$ is a path and then let $H=z_1z_2\cdots z_{h}$, where $h=|V(H)|\geq2$.
Let $G_i$ be the subgraph of $G-\{u,v\}$ that is contracted to the vertex $z_i$ of $H$, $i=1,2,\ldots,h$.
Combining with the fact that $G$ is 3-connected,
each of $u$ and $v$ has  neighbours in both $G_1$ and $G_{h}$, and if $h\geq3$,
then for each $G_j$,  at least one of $u$ and $v$ has a neighbour in it, where $2\leq j\leq  h-1$.

We assume that $h=4$, which can be proved later.
Since $H$ is path, we have $$E[V(G_1),V(G_3)]=E[V(G_1),V(G_4)]=E[V(G_2),V(G_4)]=\emptyset$$
and $$|E[V(G_1),V(G_2)]|=|E[V(G_2),V(G_3)]|=|E[V(G_3),V(G_4)]|=1.$$
Note  that each of $u$ and $v$ has  neighbours in both $G_1$ and $G_4$, and for each of $G_2$ and $G_3$, at least one of $u$ and $v$ has a neighbour in it.

Let $u_1$ and $u_4$ denote neighbours of $u$ in $G_1$ and $G_4$, respectively.
Suppose, without loss of generality, that $u$ has a neighbour $u_2$ in $G_2$.
Then we have $u_1u_4,u_2u_4\notin E(G)$.
Since $G$ is claw-free, we can derive that $u_1u_2\in E(G)$.
Since $|E[V(G_1),V(G_2)]|=1$, we have $E[V(G_1),V(G_2)]=\{u_1u_2\}$.
It follows that $N(u)\cap V(G_1)=\{u_1\}$ and $N(u)\cap V(G_2)=\{u_2\}$ since $G$ is claw-free.
If $|V(G_1)|\geq2$, then $G-\{u_1,v\}$ is disconnected with one  component contained in $G_1-u_1$, a contradiction to the fact that $G$ is 3-connected.
Thus $|V(G_1)|=1$, and then $V(G_1)=\{u_1\}$.
Consequently, we have $N(u_1)=\{u,v,u_2\}$.
Since $v$ has a neighbour in $G_4$ and $G$ is claw-free, we can obtain that $N(v)\cap V(G_2)\subseteq \{u_2\}$.
Let $a$ be an end of the only edge of $E[V(G_2),V(G_3)]$ with $a\in V(G_3)$.
If $|V(G_2)|\geq2$, then $G-\{u_2,a\}$ is disconnected  with one  component contained in $G_2-u_2$, a contradiction.
It follows  that  $V(G_2)=\{u_2\}$.
We next consider the following two cases according to whether $v$ has a neighbour in $G_3$ or not.

\begin{figure}[h]
 \centering
 \includegraphics[width=0.8\textwidth]{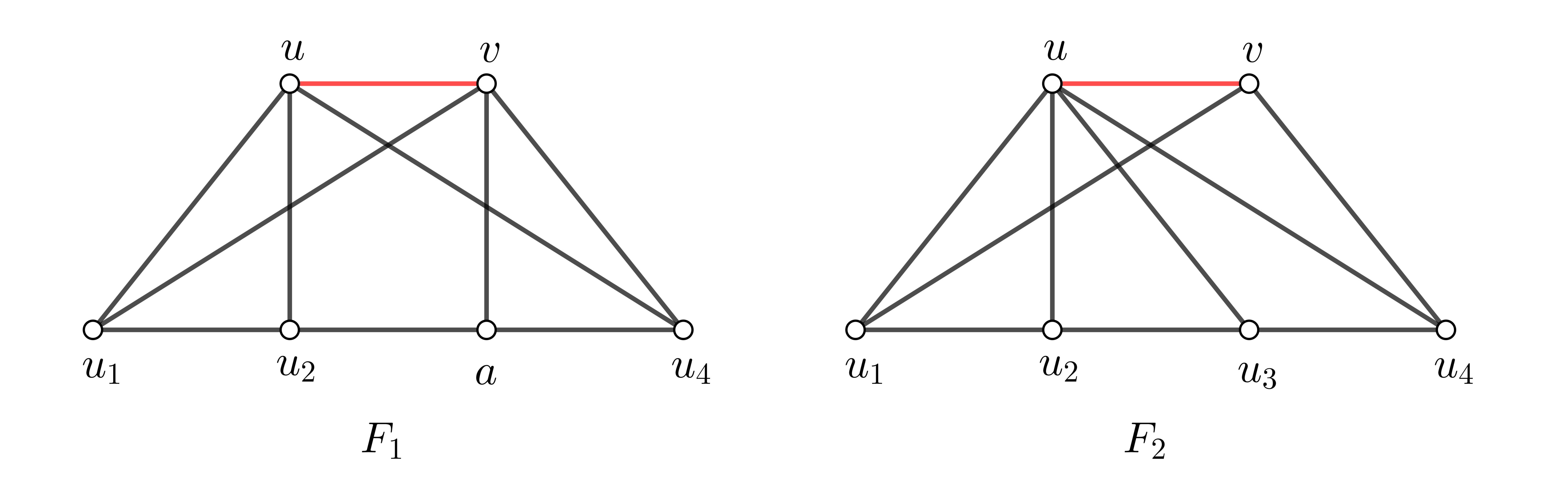}\\
 \caption{The two graphs $F_1$ and $F_2$.}\label{fig6}
\end{figure}

If the vertex $v$ has a neighbour in $G_3$, by symmetry,
we can derive that  $V(G_3)=\{a\}$ and $V(G_4)=\{u_4\}$.
Moreover, we have $N(u_4)=\{u,v,a\}$.
It follows that $F_1$ (see Figure \ref{fig6}) is a spanning subgraph of $G$.
If $ua, vu_2\notin E(G)$, then $G$ is isomorphic to $\overline{C_6}^+$.
If only one of $ua$ and $vu_2$ belongs to $E(G)$, then $G$ is isomorphic to $W_6^+$.
If $ua, vu_2\in E(G)$, then $G$ is isomorphic to $W_6^{++}$.

We now consider the vertex $v$ has no neighbour in $G_3$.
Then the vertex $u$ must have a neighbour in $G_3$, say $u_3$.
Note that $u_1u_3,u_1u_4\notin E(G)$.
Since $G$ is claw-free, we have $u_3u_4\in E(G)$, which implies that
$E[V(G_3),V(G_4)]=\{u_3u_4\}$.
Again using $G$ is claw-free, we can obtain that $N(u)\cap V(G_3)=\{u_3\}$ and $N(u)\cap V(G_4)=\{u_4\}$.
If $|V(G_3)|\geq2$, then $G-\{u_2,u_3\}$  is disconnected;
if $|V(G_4)|\geq2$, then $G-\{u_4,v\}$  is disconnected.
In both cases, we get a contradiction to the fact that $G$ is 3-connected.
It follows that  $V(G_3)=\{u_3\}$ and $V(G_4)=\{u_4\}$.
So $vu_3\notin E(G)$.
Thus $F_2$ (see Figure \ref{fig6}) is a spanning subgraph of $G$.
If $vu_2\notin E(G)$, then $G$ is isomorphic to $W_6$; otherwise, $G$ is isomorphic to $W_6^+$.
The proof of Theorem \ref{main-theorem} is completed.
\vspace{4mm}

In the following we show  $h =4$. We first show $2\leq h \leq4$, and then show $ h \neq3$ and  $ h \neq2$.

\begin{lemma}\label{2-VH-4}
$2\leq h \leq4$. 
\end{lemma}

\begin{proof}
If $ h \geq5$, by Claim \ref{H-path}, we have
$$E[V(G_1),V(G_3)]=E[V(G_1),V(G_{ h })]=E[V(G_3),V(G_{ h })]=\emptyset.$$
Recall that the vertex $u$ has a neighbour in each of $G_1$ and $G_{h}$,
and at least one of $u$ and $v$, say $u$, has a neighbour in $G_3$.
It follows that $G$ has a claw centered at the vertex $u$, a contradiction.
Hence, we have $ h \leq4$.
This, together with $ h \geq2$, implies that the result holds.
\end{proof}

\begin{lemma}\label{VH-neq3}
$ h \neq3$.
\end{lemma}

\begin{proof}
Assume, to the contrary, that $ h =3$.
Then $$E[V(G_1),V(G_3)]=\emptyset\quad  \mbox{and}\quad  |E[V(G_1),V(G_2)]|=|E[V(G_2),V(G_3)]|=1.$$
Recall that each of $u$ and $v$ has neighbours in both $G_1$ and $G_3$, and  at least one of $u$ and $v$, say $u$, has a neighbour in $G_2$. Then  each $G_i$ has a neighbour of $u$, say $u_i$, $i=1,2,3$.
We see that $u_1u_3\notin E(G)$. Using the fact that $G$ is claw-free,
we can derive that $u_2$ is adjacent to at least one of $u_1$ and $u_3$, say $u_1$.
Since $|E[V(G_1),V(G_2)]|=1$, we have $E[V(G_1),V(G_2)]=\{u_1u_2\}$.

\begin{claim}\label{claim1}
If $V(G_2)=\{u_2\}$, then $V(G_1)\neq\{u_1\}$ and $V(G_3)\neq\{u_3\}$.
\end{claim}

Assume, to the contrary, that $V(G_1)=\{u_1\}$.
Since $G$ is 3-connected and $E[V(G_1),V(G_2)]=\{u_1u_2\}$, we have $N(u_1)=\{u,v,u_2\}$.
Recall that $M$ is the unique perfect matching of $G-\{u,v\}$.
Then $u_1u_2\in M$, and hence $M\backslash\{u_1u_2\}$ is the unique perfect matching of $G_3$.
Lemma \ref{connected-with-upm} implies that $G_3$ has a bridge, a contradiction to the fact that $G_3$ is 2-edge-connected.
Hence, we have $V(G_1)\neq\{u_1\}$.
Using similar arguments, we can show that $V(G_3)\neq\{u_3\}$.
Claim \ref{claim1} holds.
\vspace{2mm}

\begin{claim}\label{claim-VG1=u1}
The vertex $u$ has precisely one neighbour $u_1$ in $G_1$ and $V(G_1)=\{u_1\}$.
\end{claim}

Assume, to the contrary, that $u$ has a neighbour $u_1'$ in $G_1$ besides the vertex $u_1$.
Then we have $u_1'u_2, u_1'u_3\notin E(G)$ since $E[V(G_1),V(G_2)]=\{u_1u_2\}$ and $E[V(G_1),V(G_3)]=\emptyset$.
Using the fact that $G$ is claw-free, we can obtain that $u_2u_3\in E(G)$, and hence $E[V(G_2),V(G_3)]=\{u_2u_3\}$. Moreover, we have $N(u)\cap V(G_2)=\{u_2\}$ and $N(u)\cap V(G_3)=\{u_3\}$.
If $|V(G_2)|\geq 2$, then $G-\{u_2,v\}$ is disconnected, contradicting the fact that $G$ is 3-connected.
Hence, we have $|V(G_2)|=1$ and then $V(G_2)=\{u_2\}$.
By Claim \ref{claim1}, we have $V(G_3)\neq\{u_3\}$ and hence $|V(G_3)|\geq 2$.
Since $E[V(G_2),V(G_3)]=\{u_2u_3\}$ and $N(u)\cap V(G_3)=\{u_3\}$, we can derive that  $G-\{u_3,v\}$ is disconnected,  a contradiction.
Therefore, $u$ has precisely one neighbour $u_1$ in $G_1$.
If $|V(G_1)|\geq 2$, then $G-\{u_1,v\}$ is disconnected, a contradiction.
Hence, we have $|V(G_1)|=1$ and then $V(G_1)=\{u_1\}$.
Claim \ref{claim-VG1=u1} holds.
\vspace{2mm}

By Claim \ref{claim-VG1=u1}, we have $V(G_1)=\{u_1\}$.
Since $v$ has a neighbour in $G_1$, we have $N(u_1)=\{u,v,u_2\}$.
We now consider the following two cases according to the neighbours of $u$ in $G_2$.
In both cases we can get contradictions, and then we prove $ h \neq3$.

If $u$ has a neighbour $u_2'$ in $G_2$ besides the vertex $u_2$, then $u_1u_2'\notin E(G)$ as $E[V(G_1),V(G_2)]=\{u_1u_2\}$.
Combining this with $u_1u_3\notin E(G)$ and $G$ is claw-free, we have $u_2'u_3\in E(G)$.
So $E[V(G_2),V(G_3)]=\{u_2'u_3\}$.
This implies that $N(u)\cap V(G_2)=\{u_2,u_2'\}$.   Using the fact that $G$ is claw-free,  we have $N(u)\cap V(G_3)=\{u_3\}$.
If $|V(G_3)|\geq 2$, then $G-\{u_3,v\}$ is disconnected, a contradiction.
So $V(G_3)=\{u_3\}$.
Since $v$ has a neighbour in $G_3$, we have $N(u_3)=\{u,v,u_2'\}$.
Since $G$ is claw-free, we conclude  that $N(v)\cap V(G_2)\subseteq\{u_2,u_2'\}$.
If $|V(G_2)|\geq3$, then $G-\{u_2,u_2'\}$ is disconnected, a contradiction.
Therefore, we have $|V(G_2)|=2$ and so $V(G_2)=\{u_2,u_2'\}$. This implies that $G_2=u_2u_2'$, a contradiction to the fact that $G_2$ is 2-edge-connected.

We now consider the case when $u$ has precisely one neighbour $u_2$ in $G_2$, i.e., $N(u)\cap V(G_2)=\{u_2\}$.
If $N(v)\cap V(G_2)\subseteq\{u_2\}$, then let $a$ be the end of the only edge of $E[V(G_2),V(G_3)]$ in $G_3$.
Since $V(G_1)=\{u_1\}$, by Claim \ref{claim1}, we have $V(G_2)\neq\{u_2\}$, which implies that $G-\{u_2,a\}$ is disconnected, a contradiction.
Therefore, $v$ has a neighbour $v_2$ in $G_2$ besides $u_2$.
Then $v_2\neq u_2$. Since $E[V(G_1),V(G_2)]=\{u_1u_2\}$, we have $u_1v_2\notin E(G)$.
Recall that $v$ has a neighbour in $G_3$, say $v_3$.
Then $u_1v_3\notin E(G)$ since $E[V(G_1),V(G_3)]=\emptyset$.
Because  $G$ is claw-free, we have $v_2v_3\in E(G)$ and $N(v)\cap V(G_3)=\{v_3\}$.
If $|V(G_3)|\geq2$, then $G-\{v_3,u\}$ is disconnected, a contradiction.
Therefore, we have $|V(G_3)|=1$ and hence $V(G_3)=\{v_3\}=\{u_3\}$.
It follows that $N(\{u,v\})\cap V(G_2)\subseteq\{u_2,v_2\}$ since $G$ is claw-free.
Because $G$ is 3-connected, $G-\{u_2,v_2\}$ is connected, which implies that $V(G_2)=\{u_2,v_2\}$.
Hence, we have $G_2=u_2v_2$, contradicting the fact that $G_2$ is 2-edge-connected.
\end{proof}

\begin{lemma}\label{VH-neq2}
$ h \neq2$.
\end{lemma}

\begin{proof}
Assume, to the contrary, that $ h =2$.
Then $E[V(G_1),V(G_2)]$ has precisely one edge, say $x_1x_2$, where $x_i\in V(G_i)$, $i=1,2$.
Since $M$ is the unique perfect matching of $G-\{u,v\}$, we have $x_1x_2\in M$, which implies that
$G_i-x_i$ has a unique perfect matching $M\cap E(G_i-x_i)$, $i=1,2$.
Hence, both $|V(G_1)|$ and $|V(G_2)|$ are odd.

\begin{claim}\label{claim-G1-x-G2-y-connected-upm}
The graphs $G_1-x_1$ and $G_2-x_2$ are both connected.
\end{claim}

Assume, to the contrary, that $G_1-x_1$ is disconnected.
Then $G_1-x_1$ has at least two distinct components, say $L_1$ and $L_2$.
Recall that $G_1$ is 2-edge-connected.
Then $x_1$ has at least one neighbours $x_{1j}$ in $L_j$, $j=1,2$.
It follows that $x_{11}x_{12}\notin E(G)$.
Since $E[V(G_1),V(G_2)]=\{x_1x_2\}$, we have $x_{11}x_2\notin E(G)$ and $x_{12}x_2\notin E(G)$, which implies that $G[\{x_1,x_{11},x_{12},x_2\}]$ is a claw centered at $x_1$ in $G$, a contradiction.
Hence, $G_1-x_1$ is connected.
Similarly, we can derive that $G_2-x_2$ is also connected.
Claim \ref{claim-G1-x-G2-y-connected-upm} holds.
\vspace{2mm}

\begin{figure}[h]
 \centering
 \includegraphics[width=0.75\textwidth]{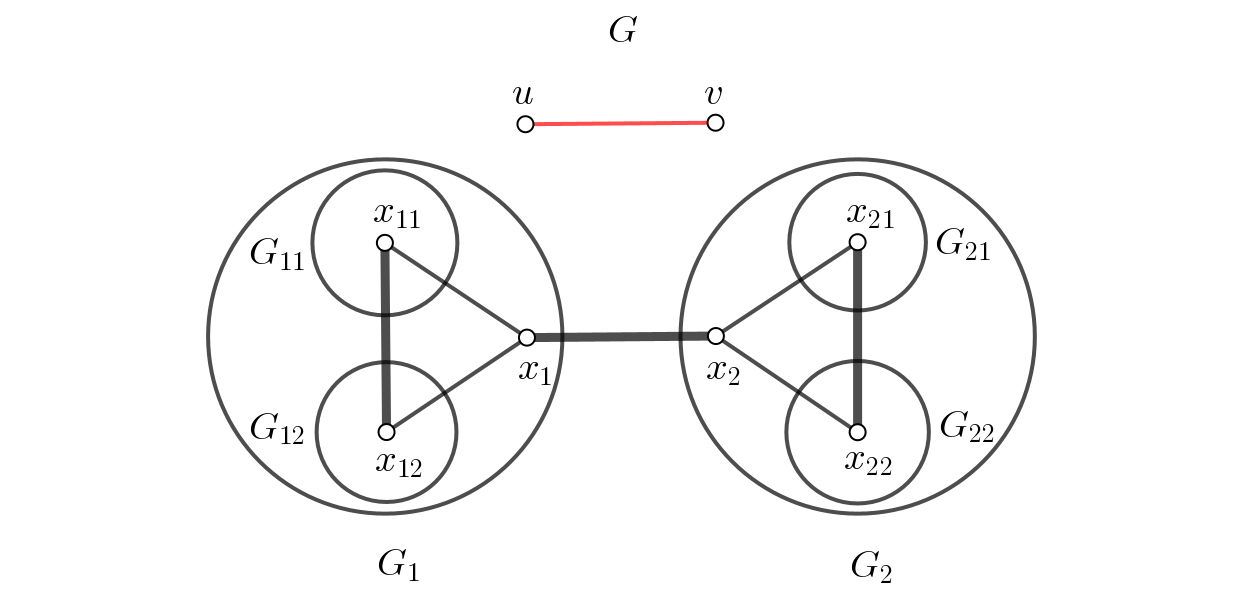}\\
 \caption{Illustration for the proof of Lemma \ref{VH-neq2}.}\label{fig4}
\end{figure}

By Claim \ref{claim-G1-x-G2-y-connected-upm}, $G_i-x_i$ is connected, $i=1,2$.
Recall that $G_i-x_i$ has a unique perfect matching $M\cap E(G_i-x_i)$.
Then, by Lemma \ref{connected-with-upm}, $G_i-x_i$ has a bridge, say $x_{i1}x_{i2}$, belonging to the unique perfect matching.
It follows that $G_i-x_i-x_{i1}x_{i2}$ has exactly two components, say $G_{i1}$ and $G_{i2}$,  where $x_{ij}\in V(G_{ij})$, $j=1,2$.
Then $E[V(G_{i1}),V(G_{i2})]=\{x_{i1}x_{i2}\}$, see  Figure \ref{fig4}.

Since $G_i$ is 2-edge-connected, the vertex $x_i$ has at least one neighbour in each of  $G_{i1}$ and $G_{i2}$.
This, together with $G$ is claw-free and $E[V(G_{i1}),V(G_{i2})]=\{x_{i1}x_{i2}\}$,
we can derive that $N(x_i)\cap V(G_{ij})=\{x_{ij}\}$, $j=1,2$.
Since $G$ is 3-connected and $E[V(G_{i1}),V(G_{i2})]=\{x_{i1}x_{i2}\}$, we see that the following claim holds.

\begin{claim}\label{claim-u-v-Gij}
 For  each $G_{ij}$ with $i,j\in\{1,2\}$, at least one of $u$ and $v$ has a neighbour in $G_{ij}$.
\end{claim}

\begin{claim}\label{claim-u-G11}
Suppose that one of  $u$ and $v$ has no neighbour in $G_{ij}$. Then
$V(G_{ij})=\{x_{ij}\}$ and $V(G_{i{(3-j)}})=\{x_{i{(3-j)}}\}$, where $i,j\in\{1,2\}$.
Moreover,
if $v$ has no neighbour in $G_{ij}$, then $N(x_{ij})=\{u,x_i,x_{i{(3-j)}}\}$ and $vx_{i{(3-j)}}\in E(G)$;
if $u$ has no neighbour in $G_{ij}$, then $N(x_{ij})=\{v,x_i,x_{i{(3-j)}}\}$ and $ux_{i{(3-j)}}\in E(G)$.
\end{claim}

We now prove Claim \ref{claim-u-G11}.
By symmetry, it suffices to prove the case that $i=j=1$.
Recall that $N(x_1)\cap V(G_{11})=\{x_{11}\}$ and $E[V(G_{11}),V(G_{12})]=\{x_{11}x_{12}\}$.
Suppose that $v$ has no neighbour in $G_{11}$.
If $|V(G_{11})|\geq2$,  then $G-\{u,x_{11}\}$ is disconnected, a contradiction.
Thus, we have $V(G_{11})=\{x_{11}\}$ and hence  $N(x_{11})=\{u,x_1,x_{12}\}$.
Recall that $u$ has a neighbour in $G_2$.
Since $G$ is claw-free, we have $N(u)\cap V(G_{12})\subseteq\{x_{12}\}$.
If $|V(G_{12})|\geq2$,  then $G-\{v,x_{12}\}$ is disconnected, a contradiction.
It follows that $V(G_{12})=\{x_{12}\}$.
If $vx_{12}\notin E(G)$, then $G-\{u,x_1\}$ is disconnected, one component is $x_{11}x_{12}$, a contradiction.
Hence, we have $vx_{12}\in E(G)$.
Using similar discussion, the result holds when $u$ has no neighbour in $G_{11}$.
Hence, Claim \ref{claim-u-G11} holds.
\vspace{2mm}

The subsequent proof will be conducted by considering the neighbours  of vertices $u$ and $v$.
Assume that there exists some $G_{ij}$, say $G_{11}$, such that $v$ has no neighbour in it.
By Claim \ref{claim-u-G11}, we have $V(G_{11})=\{x_{11}\}$, $V(G_{12})=\{x_{12}\}$, $N(x_{11})=\{u,x_1,x_{12}\}$, and $vx_{12}\in E(G)$.

We assert that each of $u$ and $v$ has neighbours in both $G_{21}$ and $G_{22}$.
If $v$ has no neighbour in some $G_{2j}$, where $j\in\{1,2\}$, then by Claim \ref{claim-u-G11}, we have $V(G_{2j})=\{x_{2j}\}$ and $ux_{2j}\in E(G)$, and then $vx_{2j}\notin E(G)$.
Recall that $x_{11}x_{2j}\notin E(G)$ and $x_{11}v\notin E(G)$.
Then the set $\{v,x_{11},x_{2j}\}$ forms a stable set of $G$, which implies that there is a claw centered at $u$ in $G$, a contradiction.
Therefore, $v$ has a neighbour in each of  $G_{21}$ and $G_{22}$.

\begin{figure}[h]
 \centering
 \includegraphics[width=\textwidth]{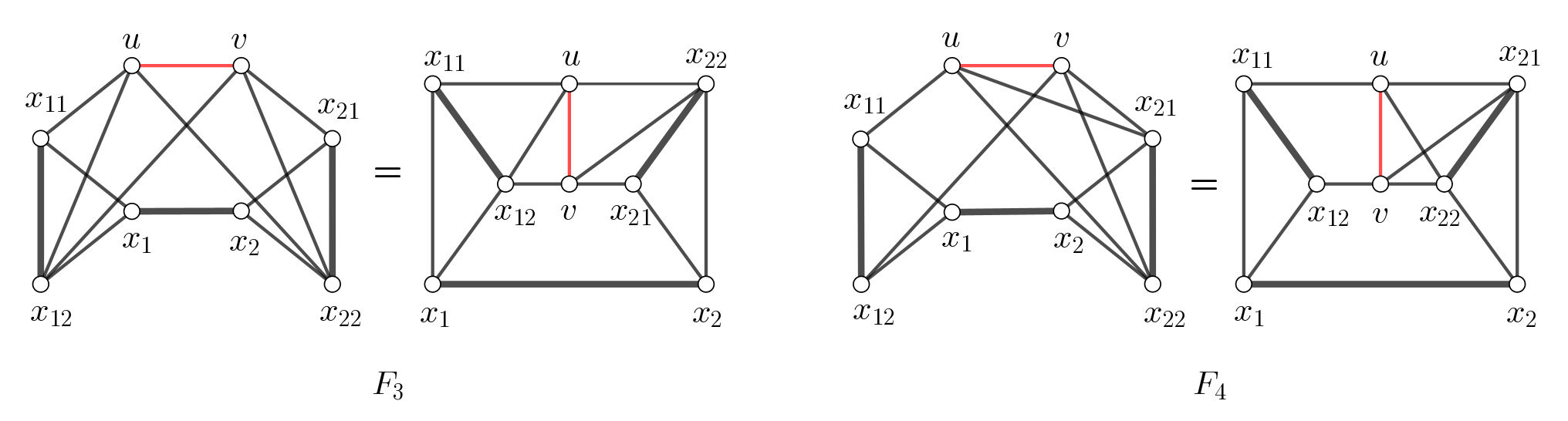}\\
 \caption{The two graphs $F_3$ and $F_4$.}\label{fig5}
\end{figure}

If $u$ has no neighbour in some $G_{2j}$, say $G_{21}$, then  by Claim \ref{claim-u-G11}, we have $V(G_{21})=\{x_{21}\}$, $V(G_{22})=\{x_{22}\}$, $N(x_{21})=\{v,x_2,x_{22}\}$, and  $ux_{22}\in E(G)$.
It follows that $ux_{21}\notin E(G)$.
Since $v$ has a neighbour in each of  $G_{21}$ and $G_{22}$, we have $vx_{21},vx_{22}\in E(G)$.
We now consider the three neighbours $u$, $x_{12}$ and $x_{21}$ of $v$ in $G$.
Since $ux_{21}\notin E(G)$, $x_{12}x_{21}\notin E(G)$ and $G$ is claw-free, we have $ux_{12}\in E(G)$.
It follows that $F_3$ (see Figure \ref{fig5}) is a spanning subgraph of $G$.
Note that $F_3-\{ux_{12},vx_{22}\}$ is isomorphic to $R_8$ (see Figure \ref{fig1}), which is a brick.
Then, by the definition of bricks, both $G$ and $G-\{ux_{12}\}$ are bricks.
It follows that $ux_{12}$ is a $b$-invariant edge of $G$ and hence it is solitary in $G$.
This implies that $G-\{u, x_{12}\}$ has a unique perfect matching.
Since $G-\{u, x_{12}, x_1, x_{11}\}$ has a hamiltonian even cycle $vx_{21}x_2x_{22}v$,
$G-\{u, x_{12}\}$ has two perfect matchings $\{x_1x_{11}, vx_{21},x_2x_{22}\}$ and $\{x_1x_{11}, x_{21}x_2,x_{22}v\}$, a contradiction.
Therefore, $u$ has a neighbour in each of  $G_{21}$ and $G_{22}$.
The assertion holds.

Using the above assertion, $G$ is claw-free and $E[V(G_{21}),V(G_{22})]=\{x_{21}x_{22}\}$,
we can derive $N(u)\cap V(G_{2j})=\{x_{2j}\}$ and $N(v)\cap V(G_{2j})=\{x_{2j}\}$, $j=1,2$.
Recall that $N(x_2)\cap V(G_{2j})=\{x_{2j}\}$.
Then $G_{2j}$ is trivial and contains only one vertex $x_{2j}$ since $G$ is 3-connected.
It follows that  $F_4$ (see Figure \ref{fig5}) is a spanning subgraph of $G$.
However, the following Claim \ref{claim-F2} asserts that this is impossible.
Consequently,  $v$ has a neighbour in each $G_{ij}$, $i,j\in\{1,2\}$.

\begin{claim}\label{claim-F2}
$F_4$ is not a spanning subgraph of $G$.
\end{claim}

If $F_4$ is a spanning subgraph of $G$, then using similar discussion as the above paragraph,
we can derive that $ux_{22}$ is a $b$-invariant edge of $G$ and hence it is solitary in $G$.
This implies that $G-\{u, x_{22}\}$ has a unique perfect matching.
Note that $G-\{u, x_{22}\}$ has a hamiltonian even cycle $x_{11}x_{12}vx_{21}x_2x_1x_{11}$.
This implies that $G-\{u, x_{22}\}$ has two perfect matchings, a contradiction.
Claim \ref{claim-F2} holds.
\vspace{2mm}

Similarly, we can show that $u$ has a neighbour in each $G_{ij}$, $i,j\in\{1,2\}$.
Since $G$ is claw-free and $E[V(G_{i1}),V(G_{i2})]=\{x_{i1}x_{i2}\}$,
we have $N(u)\cap V(G_{ij})=\{x_{ij}\}$ and $N(v)\cap V(G_{ij})=\{x_{ij}\}$.
Recall that $N(x_i)\cap V(G_{ij})=\{x_{ij}\}$.
Since $G$ is 3-connected, we can derive that $G_{ij}$ contains only one vertex $x_{ij}$,
which implies that $F_4$ is a spanning subgraph of $G$, a contradiction to Claim \ref{claim-F2}.
Therefore,  $h \neq2$.
\end{proof}

\section*{Acknowledgements}
This work is supported by the National Natural Science Foundation of China (Nos. 12501480, 12571381 and 12371361) and the Natural Science Foundation of Henan Province (No. 252300421786).


\begin{thebibliography}{1}

\bibitem{BM08}
J.A. Bondy, U.S.R. Murty, Graph Theory, Springer-Verlag, Berlin, 2008.

\bibitem{CLM2002}
M.H. Carvalho, C.L. Lucchesi, U.S.R. Murty, On a conjecture of Lov\'{a}sz concerning bricks. I. The characteristic of a matching covered graph, J. Combin. Theory Ser. B 85 (2002) 94-136.

\bibitem{CLM2012}
M.H. Carvalho, C.L. Lucchesi, U.S.R. Murty, A generalization of Little's Theorem on Pfaffian
orientations, J. Combin. Theory Ser. B 102 (2012) 1241-1266.


\bibitem{ELP82}
J. Edmonds, L. Lov\'asz, W.R. Pulleyblank, Brick decompositions and the matching
rank of graphs, Combinatorica 2 (1982) 247-274.


\bibitem{KCLL2020}
N. Kothari, M.H. Carvalho, C.L. Lucchesi, C.H.C. Little, On essentially 4-edge-connected cubic bricks, Electron. J. Combin. 27 (2020) \#P1.22.

\bibitem{Kotzig1959}
A. Kotzig, On the theory of finite graphs with a linear factor II, Mat. Fyz. \u{C}asopis.
Slovensk. Akad. Vied 9 (1959) 136-159.

\bibitem{Lovasz1987}
L. Lov\'{a}sz, Matching structure and the matching lattice, J. Combin. Theory Ser. B 43 (1987) 187-222.


\bibitem{LFW2020}
F.L. Lu, X. Feng, Y. Wang, $b$-invariant edges in essentially 4-edge-connected near-bipartite cubic bricks, Electron. J. Combin. 27 (2020) \#P1.55.


\bibitem{LM2024}
C.L. Lucchesi, U.S.R. Murty, Perfect Matchings. Springer, 2024.





\bibitem{ZLZ2025+}
Y.X Zhang, F.L Lu and H.P Zhang, Cubic bricks that every $b$-invariant edge is forcing, arXiv:2411.17295.


\end{thebibliography}
\end{document}